\documentclass[12pt]{article}

\usepackage{amsmath,amsfonts,amssymb}
\usepackage{fullpage} 
\usepackage[colorlinks]{hyperref} 
\usepackage[parfill]{parskip} 
\usepackage{enumitem} 
\usepackage{graphicx,tikz}
\usepackage{amsthm}

\begingroup
    \makeatletter
    \@for\theoremstyle:=definition,remark,plain\do{%
        \expandafter\g@addto@macro\csname th@\theoremstyle\endcsname{%
            \addtolength\thm@preskip\parskip
            }%
        }
\endgroup

\newtheorem{theorem}{Theorem}[section]

\newtheorem{lemma}[theorem]{Lemma}
\newtheorem{defi}[theorem]{Definition}
\newtheorem{claim}[theorem]{Claim}

\newcommand{\arrows}{\hookrightarrow}

\title{A Strengthening of the Erd\H{o}s--Szekeres Theorem}

\author{%
  J\'ozsef Balogh \footnote{Department of Mathematics, University of Illinois Urbana-Champaign, Urbana, Illinois 61801, USA, and Moscow Institute of Physics and Technology, Russian Federation. E-mail: \texttt{jobal@illinois.edu}. Research is partially supported by NSF Grant DMS-1764123, NSF RTG grant DMS 1937241, Arnold O. Beckman Research
Award (UIUC Campus Research Board RB 18132), the Langan Scholar Fund (UIUC), and the Simons Fellowship.}%
 \and Felix Christian Clemen \footnote {Department of Mathematics, University of Illinois Urbana-Champaign, Urbana, Illinois 61801, USA, Email: \texttt{fclemen2@illinois.edu}}%
 \and Emily Heath \footnote {Department of Mathematics, Iowa State University, Ames, Iowa 50011, USA, Email: \texttt{eheath@iastate.edu} }%
 \and Mikhail Lavrov \footnote {Department of Mathematics, Kennesaw State University, Marietta, Georgia 30067, USA, Email: \texttt{mlavrov@kennesaw.edu}}%
  }

\begin{document}
\maketitle
\begin{abstract}
The Erd\H{o}s--Szekeres Theorem stated in terms of graphs says that any red-blue coloring of the edges of the ordered complete graph $K_{rs+1}$ contains a red copy of the monotone increasing path with $r$ edges or a blue copy of the monotone increasing path with $s$ edges. 
Although $rs + 1$ is the minimum number of vertices needed for this result, not all edges of $K_{rs+1}$ are necessary. We characterize the subgraphs of $K_{rs+1}$ with this coloring property as follows: they are exactly the subgraphs that contain all the edges of a graph we call the circus tent graph $CT(r,s)$.

Additionally, we use similar proof techniques to improve upon some of the bounds on the online ordered size Ramsey number of a path given by P\'erez-Gim\'enez, Pra\l{}at, and West.
\end{abstract}

\section{Introduction}
The Erd\H{o}s--Szekeres Theorem~\cite{ESz} states that any sequence of distinct integers of length at least $rs+1$ must contain a monotone increasing subsequence of length $r+1$ or a monotone decreasing subsequence of length $s+1$.  This fundamental result in extremal combinatorics has inspired the study of many interesting variations (for example, see \cite{BST,FG, FPSS,S}).
In many of these variations, it is useful to observe that the Erd\H{o}s--Szekeres Theorem can be interpreted as a statement about ordered graphs.

An \emph{ordered graph} on $n$ vertices is a simple graph whose vertices have been labeled with $[n]=\{1, 2, \ldots, n\}$.  We denote by $P_n$ the ordered graph which is a path on $n+1$ vertices labeled $1, 2, \dots, n+1$ along the path, and by $K_n$ the ordered complete graph on $n$ vertices.

An ordered graph $G$ on $[N]$  \emph{contains} the ordered graph $H$ on $[n]$ if there is an edge-preserving injection $f:[n]\rightarrow [N]$ such that $f(i)<f(j)$ for all $1\leq i<j\leq n$. 
For ordered graphs $G,H_1,H_2$, we write $G \arrows (H_1,H_2)$ if any red-blue edge-coloring of $G$ contains a red copy of $H_1$ or a blue copy of $H_2$.

Given a sequence $a_1, \dots, a_{rs+1}$ of distinct integers, we can color the edges of $K_{rs+1}$ as follows: color an edge $ij$ with $i<j$ red if $a_i < a_j$, and blue if $a_i > a_j$. Then a monotone increasing sequence of length $r+1$ becomes a red copy of $P_r$; a monotone decreasing sequence of length $s+1$ becomes a blue copy of $P_s$. 

Not all colorings of $K_{rs+1}$ can be obtained in this way, but the Erd\H{o}s--Szekeres theorem can be strengthened to a statement about all colorings; one of its standard proofs shows that a red copy of $P_r$ or a blue copy of $P_s$ must exist in any red-blue coloring. In other words, $K_{rs+1} \arrows (P_r, P_s)$. 
 
However, $K_{rs+1}$ is not the smallest $(rs+1)$-vertex graph $G$ with the property $G\arrows (P_r,P_s)$. For example, any edge $ij$ such that $i + (rs+1-j) < \min\{r,s\}$ is not contained in any ordered path of length $\min\{r,s\}$, and therefore excluding all such edges still leaves a graph $G$ such that $G \arrows (P_r, P_s)$.
But, as we shall see, some other edges of $K_{rs+1}$ are unnecessary for less obvious reasons.

Below we define the minimal subgraph $G$ of $K_{rs+1}$ such that $G\arrows (P_r,P_s)$. Our main result is to prove a surprisingly simple characterization of all $(rs+1)$-vertex graphs $G$ with $G\arrows (P_r,P_s)$: any such ordered graph must contain our minimal example as a subgraph.  

\begin{defi}\label{def:circus-tent}
Let the \emph{circus tent graph} $CT(r,s)$  be the ordered $(rs+1)$-vertex graph with vertices $1,2,\dots, rs+1$ which is the union of the ordered $(rs+1)$-vertex graphs $G_1$ and $G_2$, defined below:
\begin{itemize}
    \item The graph $G_1$ contains an edge $ij$ iff there exists $k\in [s]$ such that $k\leq i<j\leq kr-r+2$ or $rs-kr+r\leq i<j\leq rs+2-k$.
    \item The graph $G_2$ contains an edge $ij$ iff there exists $k\in [r]$ such that $k\leq i<j\leq ks-s+2$ or $rs-ks+s\leq i<j\leq rs+2-k$.
\end{itemize}
\end{defi}

 Figure~\ref{fig:circus-tent} shows the circus tent graph $CT(3,4)$. Note that $CT(3,4)$ does not include, for example, the edge $\{2,7\}$, even though that edge is contained in many paths of length $4$.

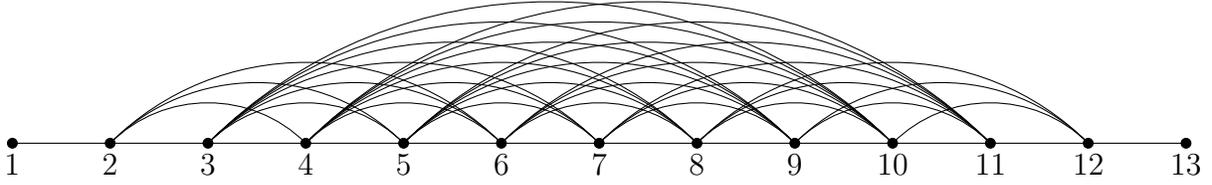
\begin{figure}[h]
\begin{center}
\begin{tikzpicture}[scale=1.3]
	\foreach \i in {1,...,13} 
	{
		\node (\i) at (\i, 0) {};
		\node [below] at (\i) {\i};
		\draw [fill=black] (\i) circle [radius=0.05]; 
	}
	\draw (1.center) -- (13.center);
	\foreach \i in {4,5,6} 
		\draw (2.center) to [out=45,in=135] (\i.center);
	\foreach \i in {5,...,10}
		\draw (3.center) to [out=45,in=135] (\i.center);
	\foreach \i in {6,...,11}
		\draw (4.center) to [out=45,in=135] (\i.center);
	\foreach \i in {7,...,11}
		\draw (5.center) to [out=45,in=135] (\i.center);
	\foreach \i in {8,...,11}
		\draw (6.center) to [out=45,in=135] (\i.center);
	\foreach \i in {9,10,11}
		\draw (7.center) to [out=45,in=135] (\i.center);
	\foreach \i in {10,11,12}
		\draw (8.center) to [out=45,in=135] (\i.center);
	\foreach \i in {11,12}
		\draw (9.center) to [out=45,in=135] (\i.center);
	\draw (10.center) to [out=45,in=135] (12.center);
	\end{tikzpicture}
\caption{The circus tent graph $CT(3,4)$.}
\label{fig:circus-tent}
\end{center}
\end{figure}

In the case $r=s$, the graphs $G_1$ and $G_2$ are identical; taking $k=r=s$ in the definition of $G_1$ gives a clique with $\binom{r^2-2r+3}{2}$ edges, and the other values of $k$ contribute $2\sum_{k=1}^{r-1} (kr-r+2 - k) = r^3-4r^2+7r-4$ edges. In total, 
\[
    |E(CT(r,r))| =\frac12 r^4 - r^3 + \frac12 r^2 + 2r -1.
\]
When $r \ne s$, both $G_1$ and $G_2$ contribute edges, and there is no exact polynomial formula for the number of edges in $CT(r,s)$. Since $CT(r,s)$ contains all edges $ij$ with $s \le i < j \le rs-r+2$, we have
\[
    \binom{rs-r-s+3}{2} \le |E(CT(r,s))| \le \binom{rs+1}{2} .
\]

We show that all edges not in  $CT(r,s)$ can be deleted from $K_{rs+1}$ and  still leave a ``good'' graph with the desired  property, while removing a single edge in $CT(r,s)$ from $K_{rs+1}$ yields a ``bad'' ordered graph without this property. 

\begin{theorem}\label{thm:circustent}
    Let $G$ be an ordered graph on $rs+1$ vertices. Then $G \arrows (P_r,P_s)$ if and only if $CT(r,s)$ is a subgraph of $G$.
\end{theorem}

This is a surprisingly simple characterization of the ``good'' ordered graphs on $rs+1$ vertices. Indeed, this is the first setting in which we have observed the phenomenon that every ``good'' graph must contain a fixed ``good'' subgraph. 

Theorem~\ref{thm:circustent} can be interpreted as a size Ramsey number problem with a fixed number of vertices. The (ordinary) \emph{size Ramsey number} of a graph $G$, denoted $\hat{r}(G)$, is the minimum integer $m$ for which there is a graph $H$ with $m$ edges such that every 2-coloring of $E(H)$ contains a monochromatic copy of $G$.

In 1983, Beck~\cite{B1} proved $\hat{r}(P_n)$ is linear in $n$, settling a question of Erd\H{o}s~\cite{E}. The current best upper bound, given by Dudek and Pra\l{}at~\cite{DP2}, is $\hat{r}(P_n)\leq 74n$, while the current best lower bound, given by Bal and DeBiasio~\cite{BD}, is $\hat{r}(P_n)\geq (3.75-o(1))n$.

Recently, this size Ramsey question has been studied in several different ordered/directed settings, including \cite{orderedbalogh,BKS,BLS,LS}. Let $\tilde r(P_r, P_s)$ denote the minimum number of edges in an ordered graph $G$ such that $G \arrows (P_r,P_s)$. This is the \emph{ordered size Ramsey number} of $P_r$ versus $P_s$.
\begin{theorem}[\cite{orderedbalogh}]
	\label{thm:size-ramsey}
	For some absolute constant $C>0$ and for all $2 \le r \le s$,
	\[
		\frac18 r^2s\le \tilde r(P_r, P_s) \le C r^2 s (\log s)^3.
	\]
\end{theorem}
 Note that the construction yielding this upper bound requires more than $rs+1$ vertices; it uses $4rs$.

Another interesting variant of the size Ramsey number is the online version introduced by Beck~\cite{B3} and by Kurek and Ruci\'nski~\cite{KR}. In the online setting, we study a game between two players, Builder and Painter. In each turn, Builder presents an edge and Painter colors it red or blue. Builder wins if Painter is forced to create a monochromatic copy of the desired graph. The minimum number of edges necessary for Builder to win is the \emph{online size Ramsey number}.
Simple arguments show that the online size Ramsey number of the path $P_n$ is at least $2n - 3$ and at most $4n-7$ ~\cite{GKP}.

This game can also be played on ordered graphs and with $t\geq 2$ colors. 
We denote by $r_o(P_{n_1},P_{n_2},\ldots,P_{n_t})$ the \emph{online ordered size Ramsey number}, which is the minimum number of edges that Builder must play in order to force Painter to create a monochromatic ordered copy of $P_{n_i}$ in some color $i$. 
In this definition, the game is played on vertex set $\mathbb N$. If we restrict the game to a fixed set of $1+\prod_{i=1}^t n_i$ vertices, then we write $r_o^*(P_{n_1},P_{n_2},\ldots,P_{n_t})$. 
For the diagonal $t$-color case when $n_1=n_i$ for all $i$, we write $r_o(P_n;t)$ and $r_o^*(P_n;t)$.
In~\cite{PPW}, P\'erez-Gim\'enez, Pra\l{}at, and West gave the following bounds on $r_o(P_n;t)$. 

\begin{theorem}[\cite{PPW}]\label{thm:doug}
    For $n \geq 2$, always $\frac{n^{t-1}}{3\sqrt{t}} \leq r_o(P_n;t) \leq tn^{t+1}$.
    \end{theorem}
    
Note that Theorem~\ref{thm:doug} is a special case of their general result on $k$-uniform hypergraphs. Below, we give a new upper bound on $r_o(P_n;t)$.
\begin{theorem}\label{thm:onlineupperbound}
	$r_o(P_{n_0}, P_{n_1}, \dots, P_{n_t})\le r_o^*(P_{n_0}, P_{n_1}, \dots, P_{n_t}) \le n_0 \prod_{i=1}^t n_i (\lfloor \log_2 n_i\rfloor+1)$.
	
	In particular, $r_o(P_n; t) = O(n^{t} (\log n)^{t-1}).$
\end{theorem}

This asymptotic bound is an improvement on Theorem~\ref{thm:doug} when $t = o(\frac{\log n}{\log \log n})$.
In the 2-color case, we also improve the lower bound from Theorem~\ref{thm:doug}, showing that $r_o(P_r,P_r)$ is superlinear in $r$.

\begin{theorem}\label{thm:onlinelowerbound}
	$r_o(P_r, P_r) \ge \log_2 (r+1)! = \Omega(r \log r)$.
\end{theorem}
The organization of the paper is the following. In Section~\ref{sec:online}, we prove Theorems~\ref{thm:onlineupperbound} and \ref{thm:onlinelowerbound} about the ordered online size Ramsey number.  In Section~\ref{sec:circustent}, we use similar techniques to prove Theorem~\ref{thm:circustent}, proving the two parts of the statement in different subsections.

\section{Online Ordered Size Ramsey Number of Paths}
\label{sec:online}
\subsection{Two-color case}

Before proving Theorem~\ref{thm:onlineupperbound}, we 
will prove a 2-color version of the result in order to give insight into our proof technique.

\begin{theorem}\label{thm:2coloronline}
	$r_o(P_r, P_s) \le  rs(\lfloor \log_2 r\rfloor+1)$.
\end{theorem}

\begin{proof}
We show that if Builder plays according to the following strategy, then at most $rs(\lfloor \log_2 r\rfloor+1)$ edges are needed to win the game. In order to determine which edges to present in each turn, Builder maintains a list of \emph{active vertices} $v_0, v_1, \dots, v_k$ for some $k < r$, satisfying the invariant that $v_i$ (if it exists) is the last vertex of a red $P_i$. Moreover, each $v_i$ will be the last vertex of a blue path of some length $b_i$.  Note that over the course of the game, $v_i$ and $b_i$ will change. (We set $b_i = -1$ if $v_i$ does not yet exist.)

Initially, Builder sets $k=0$ and lets $v_0$ be the first (leftmost) vertex.
In each round of this strategy, Builder sets $w$ to be the first vertex following all of the active vertices and plays some of the edges $v_0 w, v_1 w, \dots, v_k w$. During the round, either a new active vertex will be defined or one of the defined active vertices will be updated. The round ends when one of the following outcomes occurs. Note that in each case, Builder either increases some $b_i$ or creates a winning red path.
\begin{itemize}
\item Painter colors the edge $v_0 w$ blue. Then Builder updates the active vertex $v_0$, setting $v_0 = w$, and increases $b_0$ by $1$.

\item There is an $i$ such that Painter colors the edge $v_{i-1}w$ red and the edge $v_i w$ blue. Then Builder updates the active vertex $v_i$ by setting $v_i = w$ and increasing $b_i$ by $1$.

\item Painter colors the edge $v_k w$ red, and $k+1 < r$. Then Builder defines a new active vertex by setting $v_{k+1} = w$ and  $b_{k+1}-0$.

\item Painter colors the edge  $v_k w$ red and $k+1 = r$. Then the red path $P_{r-1}$ ending at $v_k$ together with the red edge $v_k w$ forms a red $P_r$, and Builder wins the game.
\end{itemize}
To minimize the number of edges Builder must play in this round, Builder performs  a procedure similar to binary search on the set of edges $\{v_0 w, v_1 w, \dots, v_k w\}$. Builder begins by playing the middle edge $v_{\lfloor k/2 \rfloor} w$. If Painter colors this edge red, then Builder cuts the set of edges in half and continues this process on the second half of the edges, $\{v_{\lfloor k/2 \rfloor + 1}w, \dots, v_k w\}$. If Painter colors $v_{\lfloor k/2 \rfloor} w$ blue, then Builder cuts the set of edges in half and instead continues this process on the first half of the edges, $\{v_0w, \dots, v_{\lfloor k/2 -1 \rfloor} w\}$.

After at most $\lfloor \log_2 (k+1) \rfloor$ turns, the set of edges remaining is a single edge $\{v_i w\}$. Moreover:
\begin{itemize}
\item If $i>0$, then at some point in this round, $v_i w$ was the first edge in the second half of a subset, and $v_{i-1}w$ was colored red.
\item If $i<k$, then at some point, $v_i w$ was the last edge in the first half of a subset, and $v_{i+1}w$ was colored blue.
\end{itemize}
Under these conditions, no matter how Painter colors the edge $v_iw$, one of Builder's goals is satisfied: Builder either increases $b_i$, or increases $b_{i+1}$, or obtains a red $P_r$. Thus, each round can be completed in $\lfloor \log_2 (k+1) \rfloor + 1 \le \lfloor \log_2 r \rfloor + 1$ turns.

We can track the progress of this strategy by considering the ordered $r$-tuple $(b_0, b_1, \dots, b_{r-1})$, which starts at $(0, -1, -1, \dots, -1)$. After $rs$ rounds, we have $\sum_{i=0}^{r-1} b_i=r(s-1)+1$, so $b_i \ge s$ for some $i$, and Builder wins. Therefore, Builder needs at most $rs$ rounds to win according to this strategy, for a total of $rs(\lfloor \log_2 r \rfloor + 1)$ turns.
\end{proof}

\subsection{Proof of Theorem~\ref{thm:onlineupperbound}}

In order to prove our upper bound in the multicolor case, we replace the binary search procedure used by Builder in the proof of Theorem~\ref{thm:2coloronline} with a multidimensional search procedure.

\begin{proof}
Builder's strategy, played on vertices $1, 2, \dots,1+ \prod_{i=0}^t n_i$, is as follows. Throughout the game, Builder maintains a $t$-dimensional array of active vertices, labeled by integer points $\mathbf x = (x_1, \dots, x_t) \in [0,n_1) \times \ldots \times [0, n_t)$. Let $v(\mathbf x)$ be the active vertex labeled by $\mathbf x$. An active vertex may be undefined; however, if $v(\mathbf x)$ is defined, then for each $i=1, \dots, t$, there is an ordered path in color $i$ of length $x_i$ ending at $v(\mathbf x)$. Initially, $v(\mathbf 0) = 1$, and no other active vertices are defined.

Let $\ell(\mathbf x)$ denote the length of the longest ordered path in color $0$ ending at $v(\mathbf x)$. If this vertex is undefined, then we say $\ell(\mathbf x) = -1$.

Builder plays in rounds of length at most $\prod_{i=1}^t \lfloor\log_2 n_i\rfloor+1$. In each round, Builder plays edges from some of the defined active vertices to $w$, the first vertex following all of the active vertices. At the conclusion of a round, Builder either wins the game, or updates some active vertex $v(\mathbf x)$ by setting $v(\mathbf x)=w$, which increases $\ell(\mathbf x)$ by $1$.

Throughout the game, Builder uses the following  \emph{$d$-dimensional search procedure} to choose which edges to play.

Let $S(y_{d+1}, y_{d+2}, \dots, y_t)$ denote the set
\[
    S(y_{d+1}, y_{d+2}, \dots, y_t) = [0, n_1) \times \dots \times [0, n_d) \times \{y_{d+1}\} \times \dots \times \{y_t\}.
\]
Builder applies the $d$-dimensional search procedure to a set $S(\mathbf y)$ for some $\mathbf y \in \mathbb Z^{t-d}$ in order to obtain one of the following outcomes:
\begin{enumerate}
\item A point $\mathbf x \in S(\mathbf y)$ such that the edge $v(\mathbf x)w$ has some color $i > d$, or
\item A point $\mathbf x \in S(\mathbf y)$ such that either the active vertex $v(\mathbf x)$ is undefined or the edge $v(\mathbf x)w$ has color $0$. Moreover, for each $j=1, \dots, d$, either $x_j = 0$ or there is an \emph{assistant point} $\mathbf x^{(j)} \in S(\mathbf y)$ such that $x_j^{(j)} = x_j - 1$ and the edge $v(\mathbf x^{(j)})w$ has color $j$.
\end{enumerate} 
This procedure is defined recursively. In the $0$-dimensional search procedure on $S(\mathbf y)$, Builder draws the edge from $v(\mathbf y)$ to $w$, if the active vertex $v(\mathbf y)$ is defined.  If Painter colors this edge with color $0$, or if $v(\mathbf y)$ is undefined, then Builder obtains outcome 2 by setting $\mathbf x = \mathbf y$. If Painter colors this edge with some color $i \ge 1$, then Builder obtains outcome 1 by setting $\mathbf x = \mathbf y$.

For $d\ge 1$, the $d$-dimensional search procedure on $S(\mathbf y)$ is similar to a binary search. It uses an interval $[a,b)$ initially set to $[0, n_d)$. To cut the interval in half, it performs the $(d-1)$-dimensional search procedure on
\(
    S(\lfloor \tfrac{a+b}{2} \rfloor, \mathbf y) \subset S(\mathbf y).
\)

When this subprocedure is done, there are three possibilities:
\begin{itemize}
\item If the subprocedure yields outcome 1 and the edge $v(\mathbf x)w$ has color $d$, then the procedure continues with interval $[\lfloor \tfrac{a+b}{2} \rfloor + 1, b)$.

\item If the subprocedure yields outcome 1 and the edge $v(\mathbf x)w$ has color $i>d$, then the procedure terminates, having also obtained outcome 1 with the same $\mathbf x$.

\item If the subprocedure yields outcome 2, then the procedure continues with interval $[a, \lfloor \tfrac{a+b}{2} \rfloor)$.
\end{itemize}
After at most $\lfloor \log_2 n_d \rfloor+1$ steps of the $(d-1)$-dimensional search procedure, Builder is left with the empty interval $[a,a)$. 
\begin{itemize}
\item If $a=0$, then the $(d-1)$-dimensional search procedure was performed on $S(0, \mathbf y)$ and yielded outcome 2 with some point $\mathbf x \in S(0, \mathbf y)$. Then the $d$-dimensional search procedure will yield outcome 2 with the same $\mathbf x$ and the same assistant points; since $x_d = 0$, no assistant point $\mathbf x^{(d)}$ is necessary.
\item If $a=n_d$, then the $(d-1)$-dimensional search procedure was performed on $S(n_d-1, \mathbf y)$ and yielded outcome 1 with some point $\mathbf x \in S(n_d-1, \mathbf y)$ such that the edge $v(\mathbf x)w$ has color $d$. Then there is an ordered path of length $n_d$ in color $d$ ending at $w$: the path of length $n_d - 1$ ending at $v(\mathbf x)$ followed by the edge $v(\mathbf x)w$, and Builder wins.
\item If $0 < a < n_d$, then the $(d-1)$-dimensional search procedure was performed on $S(a, \mathbf y)$ and yielded outcome 2 with some point $\mathbf x \in S(a, \mathbf y)$; it was also performed with $S(a-1, \mathbf y)$ and yielded outcome 1 with some point $\mathbf x' \in S(a-1, \mathbf y)$ such that the edge $v(\mathbf x') w$ has color $d$. 

In this case, the $d$-dimensional search procedure can yield outcome 2 with $\mathbf x$, taking the same assistant points and adding the assistant point $\mathbf x^{(d)} = \mathbf x'$, which satisfies the conditions required.
\end{itemize}

By induction, the $d$-dimensional search procedure takes at most $\prod_{i=1}^d (\lfloor \log_2 n_i \rfloor+1)$ moves. 

A round of Builder's strategy consists of performing the $t$-dimensional search procedure on the entire set $[0, n_1) \times \dots \times [0,n_t)$. This search procedure never yields outcome 1, since no color $i>t$ is available. Therefore, Builder obtains outcome 2 with some point $\mathbf x$.

We claim that there is an ordered path of length $x_i$ in color $i$ ending at $w$ for each $i=1, \dots, t$, and therefore $w$ satisfies the prerequisites for replacing $v(\mathbf x)$ as an active vertex. This is automatic if $x_i = 0$. If $x_i > 0$, then there is an assistant point $\mathbf x^{(i)}$ such that $x_i^{(i)} = x_i - 1$ and the edge $v(\mathbf x^{(i)})w$ has color $i$. Then the ordered path of length $x_i - 1$ in color $i$ ending at $v(\mathbf x^{(i)})$ can be followed by the edge $v(\mathbf x^{(i)})w$ to get the path of length $x_i$ Builder wants. 

If $v(\mathbf x)$ is undefined, Builder defines the active vertex  $v(\mathbf x) = w$ and increases $\ell(\mathbf x)$ from $-1$ to $0$. Otherwise, the edge $v(\mathbf x) w$ has color $0$, and there is an ordered path of length $\ell(\mathbf x) +1$ in color $0$ ending at $w$: the path of length $\ell(\mathbf x)$ ending at $v(\mathbf x)$, followed by the edge $v(\mathbf x)w$. Then Builder updates $v(\mathbf x)$ by setting $v(\mathbf x) = w$, which increases $\ell(\mathbf x)$ by $1$.

After $n_0 n_1 \cdot \ldots \cdot  n_t$ rounds, either $\ell(\mathbf 0)$ has been increased $n_0$ times (from $0$ to $n_0$) or one of the other $n_1 \cdot \ldots \cdot  n_t-1$ values $\ell(\mathbf x)$ has been increased $n_0 + 1$ times (from $-1$ to $n_0$). In either case, there is an ordered path of length $n_0$ in color $0$, and Builder wins. This strategy requires at most $n_0 \prod_{i=1}^t n_i (\lfloor \log_2 n_i\rfloor+1)$ moves, as desired.
\end{proof}

\subsection{Proof of Theorem~\ref{thm:onlinelowerbound}}

Our proof of Theorem~\ref{thm:onlinelowerbound} uses  the following lemma, which gives a general strategy for finding lower bounds on $r_o(P_r,P_s)$. 

\begin{lemma}\label{coloring-set}
	Let $\mathcal C = \{C_1, C_2, \dots, C_N\}$ be a set of edge-colorings of $K_{rs+1}$. Let $p > 0$ be such that for every $P_r$ in $K_{rs+1}$, there are at most $pN$ colorings in $\mathcal C$ in which the path is completely red, and for every $P_s$ in $K_{rs+1}$, there are at most $pN$ colorings in $\mathcal C$ in which the path is completely blue. 
	
	Then $r_o^*(P_r, P_s) \ge \log_2 \frac 1p$.
	
	Moreover, if we can find a set of colorings of $K_n$ with the same ratio $p$ for sufficiently large $n$,	then $r_o(P_r, P_s) \ge \log_2 \frac 1p$ as well.
\end{lemma}
\begin{proof}
Painter's strategy for playing on $rs+1$ vertices is as follows. After $k$ edges have been played and colored, Painter computes $\mathcal C_k \subseteq \mathcal C$, consisting of all colorings in $\mathcal C$ which agree with the partial coloring of $K_{rs+1}$ built so far.

When Builder plays an edge $vw$, Painter splits $\mathcal C_k$ into two sets: $\mathcal C_k^r$, consisting of all colorings in which $vw$ is red, and $\mathcal C_k^b$, in which $vw$ is blue. Painter colors edge $vw$ red (so that $\mathcal C_{k+1} = \mathcal C_k^r$) if $|\mathcal C_k^r| \ge |\mathcal C_k^b|$, and colors edge $vw$ blue (so that $\mathcal C_{k+1} = \mathcal C_k^b$) otherwise. Thus, at each step, Painter ensures that $|\mathcal C_{k+1}| \ge \frac12 |\mathcal C_k|$; by induction, $|\mathcal C_k| \ge 2^{-k}N$.

If Painter loses the game because a red $P_r$ or a blue $P_s$ has been created, then by definition of $p$, $|\mathcal C_k| \le pN$. Therefore $p \ge 2^{-k}$, or $k \ge \log_2 \frac 1p$.

This argument bounds $r_o^*(P_r, P_s)$; to prove a bound of $r_o(P_r, P_s) \ge k$ using this lemma, Painter simulates Builder's moves on a graph with $n \ge 8^k$ vertices. For every edge Builder plays, Painter plays an edge in the simulation so that the graphs in the actual graph and in Painter's simulation are order-isomorphic except possibly for isolated vertices. Moreover, Painter makes sure that after the $i^{\text{th}}$ move, there are at least $8^{k-i}$ isolated vertices between any two vertices with positive degree in Painter's simulation. This will always be possible and guarantees that Painter can always simulate Builder's future moves for at least $k$ steps. 

Then, Painter determines a color for the edge in the simulated graph, using a collection $\mathcal C$ of colorings of $K_n$. Painter uses that color to play in the actual graph.
\end{proof}

As a corollary, we obtain Theorem~\ref{thm:onlinelowerbound}, which gives an improved lower bound in the diagonal case.
Note that applying Lemma~\ref{coloring-set} with $\mathcal C$ consisting of all possible red-blue edge-colorings of  $K_{r^2+1}$ will give a weaker bound on $r_o(P_r,P_r)$ than we want. Our improvement comes from selecting an appropriate  subfamily of colorings.

\begin{proof}
Choosing $n$ to be as large as necessary, apply Lemma~\ref{coloring-set} with the following set of colorings of $K_n$: for every permutation $\sigma$ of $\{1, 2, \dots, n\}$, take the coloring in which edge $ij$ is red if $\sigma(i) < \sigma(j)$ and blue if $\sigma(i) > \sigma(j)$.

Then $p = \frac{1}{(r+1)!}$, because asking for a specific path $P_r$ to be monochromatic requires $r+1$ values of $\sigma$ to be in a specific relative order. Therefore $r_o(P_r, P_r) \ge \log_2 \frac 1p = \log_2 (r+1)!$.
\end{proof}

\section{The Circus Tent Theorem}
\label{sec:circustent}

In this section, we prove our main result, Theorem~\ref{thm:circustent}.  In Subsection~\ref{sub:necessary}, we show that every subgraph $G$ of $K_{rs+1}$ such that $G\arrows (P_r,P_s)$ must contain $CT(r,s)$. Then, in Subsection~\ref{sub:sufficient}, we show that $CT(r,s)\arrows$ $(P_r,P_s)$. 

\subsection{A Circus Tent is necessary}\label{sub:necessary}

\begin{lemma}\label{lemma:necessary}
For every edge $e\in E(CT(r,s))$, $K_{rs+1}-e$ has a red-blue edge-coloring without any ordered red path of length $r$ or blue path of length $s$. 
\end{lemma}

\begin{proof} 
Let $e=ij\in E(CT(r,s))$ with $i<j$. Then either $e\in E(G_1)$ or $e\in E(G_2)$. Without loss of generality, assume $e\in E(G_1)$. (For $e\in E(G_2)$, just switch the roles of $r$ and $s$ throughout the proof.) Then there exists some $k\in [s]$ such that $k\leq i<j\leq kr-r+2$ or $rs-kr+r\leq i<j\leq rs+2-k$. 

Note that by symmetry, it suffices to prove the result for edges of the first type. Indeed, if $K_{rs+1}-ij$ has an edge-coloring with no red $P_r$ or blue $P_s$, then its mirror image is such an edge-coloring for $K_{rs+1}-i'j'$, where $i'=(rs+2)-j$ and $j'=(rs+2)-i$. So, let us assume that $e=ij$ where $k\leq i<j\leq kr-r+2$ for some fixed $k\in [s]$.

First, we will provide a vertex-labeling of $K_{rs+1}-e$ which will be used to construct the desired edge-coloring. Our definitions will require some additional notation. Given a set $S$ of ordered pairs, let $S^*$ be the sequence formed by taking the elements of $S$ in lexicographic order. Let $S^* \oplus T^*$  denote the sequence formed by  concatenating $S^*$ and $T^*$.

Define the following (possibly empty) sets of labels:
\begin{align*}
    X = &\ \{(0,y):0\leq y \leq k-2\}, \\
    Y = &\ \{(x,y):1\leq x\leq r-1, 0\leq y\leq k-2\}, \\
    Z = &\ \{(x,y):0\leq x\leq r-1, k-1\leq y\leq s-1\}\setminus\{(0,k-1)\}.
\end{align*}
Assign the labels from the sequence $X^* \oplus Y^* \oplus Z^*$ to the vertices $[rs+1] \setminus \{i,j\}$ in order; assign the label $(0,k-1)$ to the vertices $i$ and $j$. Note that all vertices between $i$ and $j$ receive a label from $Y$.

Observe that the resulting vertex-labeling has the property that whenever $(a,b)$ comes before $(c,d)$, either $a<c$ or $b<d$ (or both), with one exception: the two copies of $(0,k-1)$, which appear on the vertices $i$ and $j$.
For pairs of labels other than $(0,k-1)$, this holds by construction. For pairs of labels involving $(0,k-1)$, this holds because the labels from $X$ all appear on vertices to the left of $i$, while the labels from $Z$ all appear on vertices to the right of $j$.

Now color $G = K_{rs+1} - e$ as follows. For any edge of $G$, consider the two corresponding labels $(a,b)$ and $(c,d)$ in the sequence; if $a<c$, color the edge red, and if $a \ge c$ but $b < d$, color the edge blue. This provides a red-blue edge-coloring of $G$ without creating a red $P_r$ or blue $P_s$, since the first coordinate can only increase from 0 to $r-1$ and the second can only increase from 0 to $s-1$. 
\end{proof}

Lemma~\ref{lemma:necessary} implies half of the statement of Theorem~\ref{thm:circustent}; if $G$ is an ordered graph on $rs+1$ vertices and $G \arrows (P_r,P_s)$, then $CT(r,s) \subseteq G$.

\subsection{A Circus Tent is sufficient}\label{sub:sufficient}

\begin{lemma}
We have $CT(r,s) \arrows (P_r,P_s)$.
\end{lemma}

In order to prove this result, we give a strategy for Builder in the online game on $rs+1$ vertices which is a slight modification of the strategy used to prove Theorem~\ref{thm:2coloronline}. Then, we argue that Builder will win by applying this strategy without ever playing an edge outside the circus tent graph $CT(r,s)$. This implies that Painter cannot have an offline strategy for coloring $CT(r,s)$ without a red $P_r$ or blue $P_s$, or else Painter could have used that strategy for the online game as well.

As in the earlier strategy, Builder maintains a list of vertices $v_0, v_1, \dots, v_{r-1}$ and a corresponding tuple $(b_0, b_1, \dots, b_{r-1})$ throughout the game with the property that $v_i$ is the rightmost vertex of a red path of length $i$ and a blue path of length $b_i$. Builder's strategy will proceed in $rs+1$ stages labeled $1, \dots, rs+1$; we will write $v_i(t)$ and $b_i(t)$ for the values of $v_i$ and $b_i$ respectively after stage $t$ is completed. Some of these vertices $v_i(t)$ may be undefined (in which case we set $b_i(t) = -1$), and some of these vertices may be the same. At the beginning of the strategy, which we represent by $t=0$, $v_i(0)$ will be undefined for all $i$.

In stage $t$ of the strategy, Builder asks Painter to color the edges $v_i(t-1)t$ for every defined vertex $v_i(t-1)$. 
Recall that the strategy used to prove Theorem~\ref{thm:2coloronline} requires Builder to use a binary search procedure to minimize the number of edges played in the game. 
Since we are not interested in minimizing the number edges played in the current proof, our new strategy allows Builder to draw all of the edges $v_i(t-1)t$ in stage $t$.
 
As before, the vertices $v_0, v_1, \dots, v_{r-1}$ and the tuple $(b_0, b_1, \dots, b_{r-1})$ are updated at the end of stage $t$, according to the colors Painter assigns to the edges $v_i(t-1)t$.  

\begin{itemize}
\item If $i<r$ is the least nonnegative integer such that either the edge $v_i(t-1) t$ is blue or the vertex $v_i(t-1)$ is undefined, then
Builder updates the vertex $v_i(t)$ by setting $v_i(t) = t$ and $b_i(t) = b_i(t-1) + 1$.

\item If the vertex $v_i(t-1)$ is defined and the edge $v_i(t-1)t$ is red for all $i<r$, then there is a red path of length $r$ ending at $t$: the path of length $r-1$ ending at the vertex $v_{r-1}(t-1)$, followed by the red edge $v_{r-1}(t-1)t$. In this case, Builder wins.

\item Our new addition to Builder's strategy is the following post-processing step: After updating the vertex $v_i(t)$ to $t$, Builder also updates any vertex $v_j(t)$ with $j<i$ and $b_j(t-1) \le b_i(t)$, setting $v_j(t) = t$ and  $b_j(t) = b_i(t)$. 
 
These updates preserve Builder's invariant because $t$ is the rightmost vertex both of a red path of length $j$ (it is actually the rightmost vertex of a red path of length $i$, and $i>j$) and a blue path of length $b_i(t)$, so we can set $b_j(t) = b_i(t)$.

In all other cases, we keep the vertex $v_j(t) = v_j(t-1)$ and $b_j(t) = b_j(t-1)$.
\end{itemize}

The proof that Builder's strategy always works is the same as before, so we now show that Builder's strategy never uses edges outside $CT(r,s)$ in three steps. We will assume without loss of generality that $r \le s$. 

In the first step, we consider edges starting at vertices $1, 2, \dots, r$. For each positive $k \le r$, vertex $k$ has an edge to vertices $k+1, k+2, \dots, ks-s+2$ in $G_2$ and therefore in $CT(r,s)$. The following claim shows that no other edges starting at vertices $1, 2, \dots, r$ are used in Builder's strategy:

\begin{claim}\label{claim:early-game}
	For positive $k \le r$, if $v_i(t) \le k$ for any $i$ and $t$, then $t \le ks-s+1$.
\end{claim}

We postpone the proof of this claim, and the subsequent technical claims, to the next section.

In the second step, we consider edges starting at vertices $r+1, r+2, \dots, s$. Setting $k=s$ in Definition~\ref{def:circus-tent} shows that $G_1$ contains all of the edges $ij$ for  $r \le i < j \le (r-1)s+2$. The following claim shows that this set contains all edges used in Builder's strategy which start at vertices $r+1, r+2, \dots, s$:

\begin{claim}\label{claim:middle-game}
	For $r+1 \le k \le s$, if $v_i(t) \le k$ for any $i$ and $t$, then $t \le (r-1)s+1$.
\end{claim}

In the third step, we consider edges starting at all other vertices. In Definition~\ref{def:circus-tent}, taking $k=r$, we see that $G_2$ has edge $ij$ for $s \le i < j \le rs-r+2$. This shows that if Builder draws an edge $ij$ with $i > s$, then that edge certainly exists in $CT(r,s)$ unless $j \ge rs-r+3$. 

To handle edges that do end at a vertex $j \ge rs-r+3$, we use the other cliques in $G_2$; taking $k=1, \dots, r-1$ in Definition~\ref{def:circus-tent}, we see that $G_2$ has edges from each $i \ge rs-ks+s$ to $j = rs-k+2$. No other edges ending at $j$ will be used if, after stage $t = rs-k+1$, we have $v_i(t) \ge rs-ks+s$. This is guaranteed by the following claim:

\begin{claim}\label{claim:late-game}
	For positive $k \le r-1$, if Builder has not won by stage $t_k = rs-k+1$, then $v_i(t_k) \ge rs-ks+s$ for all $i$.
\end{claim}

Since these three claims prove that Builder will win the  game using only edges from $CT(r,s)$, we conclude that there is no coloring of $CT(r,s)$ avoiding both a red $P_r$ and blue $P_s$. That is, $CT(r,s)\arrows (P_r,P_s)$, as desired.

\subsection{Proofs of technical claims}

We begin with the following lemma:

\begin{lemma}\label{lemma:increments}
Suppose that in stage $t^*$, Builder sets $v_i(t^*) = t^*$ and $b_i(t^*) = b^*$. If  we still have $v_i(t) = t^*$ after stage $t > t^*$, then we must have
\[
	t \le t^* + i(s-1-b^*) + (r-1-i)b^*.
\]
\end{lemma}
\begin{proof}
For each $j<i$, we have $b_j(t^*) \ge b^*$. There can be at most $s-1-b^*$ stages at which $b_j$ increases before Builder's victory, because $b_j(t) \le s-1$. Altogether, there are at most $i(s-1-b^*)$ stages in the interval $(t^*, t]$ at which any $b_j$ for $j<i$ is increased.

For each $j>i$, we have $b_j(t^*) \le b^*$, and in order to have $v_i(t) = t^*$, one of two possibilities must hold:
\begin{itemize}
\item $b_j(t) = b_j(t^*) = b^*$, and $b_j$ is never increased.

\item $b_j(t) < b^*$, and $b_j$ can be increased at most $b^*$ times: starting from $-1$ to at most $b^*-1$.
\end{itemize}
Altogether, there are at most $(r-1-i)b^*$ stages in the interval $(t^*, t]$ at which any $b_j$ for $j>i$ is increased.

However, at least one $b_j$ must increase at each stage in the interval $(t^*, t]$. Therefore the number of stages, $t - t^*$, is at most $i(s-1-b^*) + (r-1-i)b^*$, proving the lemma.
\end{proof}

\begin{proof}[Proof of Claim~\ref{claim:early-game}]

We begin by showing that at each stage $t$, for every $i$ such that $v_i(t)$ is defined, we have $i + b_i(t) \le t-1$.

To show this, we induct on $t$. When $t=0$, the claim holds trivially: none of the vertices $v_i(0)$ are defined.

At stage $t$, we either define a new $v_i(t)$ and set $b_i(t) = 0$, or set $b_i(t) = b_i(t-1) + 1$ for some $i$. In the first case, we must have $i \le t-1$, since only $t$ vertices are considered in step $t$, so no red path of length $t$ or greater can be found. In the second case, we have $i + b_i(t-1) \le t-2$ by the inductive hypothesis, so $i + b_i(t) \le t-1$. 

Finally, $b_j(t)$ may change for some values of $j< i$ in the ``post-processing" step, when we set $b_j(t) = b_i(t)$ if $j<i$ and $b_j(t-1) \le b_i(t)$. However, this step cannot cause $j + b_j(t)$ to violate the inequality, because $j + b_j(t) < i + b_j(t) = i + b_i(t) \le t-1$.

Now we are ready to proceed to the main proof. Take a positive $k \le r$, and suppose $t^*= v_i(t) \le k$; our goal is to show $t \le ks-s+1$. Let $b^* = b_i(t^*)$; by our observation earlier, $i + b^* \le t^* - 1 \le k-1$.

By Lemma~\ref{lemma:increments}, if $v_i(t) = t^*$, then
\[
	t \le t^* + i(s-1-b^*) + (r-1-i)b^*.
\]
Because $r \le s$ and $b^* \ge 0$, we have $(r-1-i)b^* \le (s-1-i)b^*$. Therefore
\begin{align*}
t	&\le t^* + i(s-1-b^*) + (s-1-i)b^* = t^* + (i + b^*)(s-1) - 2ib^* \\
	&\le k + (k-1)(s-1) - 0 = (k-1)s + 1.
\end{align*}
This completes the proof.
\end{proof}

\begin{proof}[Proof of Claim~\ref{claim:middle-game}]
Take any $k \in (r,s]$, and suppose $t^*=v_i(t)  \le k$; let $b^* = b_i(t^*)$. Our goal is to show that $t \le (r-1)s+1$.

By Lemma~\ref{lemma:increments},
\[
	t \le t^* + i(s-1-b^*) + (r-1-i)b^*.
\]
We must have $0 \le b^* \le s-1$. For any fixed $b^*$, the right-hand side of this inequality is linear in $i$, and we have $0 \le i \le r-1$. Therefore the expression is maximized either when $i=0$ and it is $t^* + (r-1)b^* \le t^*+(r-1)(s-1)$, or when $i=r-1$ and it is $t^*+(r-1)(s-1-b^*) \le t^*+(r-1)(s-1)$. 

In both cases,
\[
	t \le t^* + (r-1)(s-1) \le s + (r-1)(s-1) \le (r-1)s+1,
\]
proving the claim.
\end{proof}

\begin{proof}[Proof of Claim~\ref{claim:late-game}]
Let $t_k = rs-k+1$. Let $t^* = v_i(t_k)$ and $b^* = b_i(t^*)$. Note that at every stage $t$ when $v_i$ or $b_i$ change, we set $v_i(t) = t$; therefore we also have $b^* = b_i(t)$ for all $t \in [t^*, t_k]$. Our goal is to show that $t^* \ge rs-ks+s = t_k - (k-1)(s-1)$ or, equivalently,
\begin{equation}
	\label{eq:late-game}
	t_k \le t^* + (k-1)(s-1).
\end{equation}
The sum $\sum_{j=0}^{r-1} b_j(t)$ starts at $-r$ when $t=0$; at each stage, it increases by at least $1$. Therefore at stage $t_k$, it must satisfy
\[
	\sum_{j=0}^{r-1} b_j(t_k) \ge t_k - r = r(s-1)-k+1.
\]
Because $s-1 \ge b_0(t_k) \ge \ldots \ge b_{r-1}(t_k)$, we also have
\[
	j(s-1) + (r-j) b_j(t_k) \ge r(s-1)-k+1
\]
for any $j$, which can be rewritten as
\begin{equation}
	\label{eq:product-bound}
	(r-j)(s-1 - b_j(t_k)) \le k-1.
\end{equation}
We complete the proof by considering two cases.

\textbf{Case 1:} $i \le r-k$.

In relation~\eqref{eq:product-bound}, when $j \le r-k$, the first factor on the left-hand side exceeds $k$, and therefore the second factor must be $0$. Therefore $b_j(t_k) = s-1$ for all $j \le r-k$. In particular, $b^* = s-1$.

We must have $b_j(t^*) = s-1$ for $j < i$ by monotonicity. However, we must also have $b_j(t^*) = s-1$ for $i < j \le r-k$, since  $b_j(t_k) = s-1$ for such $j$, and if $b_j(t)$ was updated to $s-1$ at some stage $t \in (t^*, t_k]$, then $v_i(t)$ would also be updated in the post-processing step. In that case, we would have $v_i(t_k) = t > t^*$, contrary to our assumption.

Therefore at stages $t \in (t^*, t_k]$, only $b_j(t)$ for $t = r-k+1, \dots, r-1$ can be updated. Each of them can be updated at most $s-1$ times: none of them can reach $s-1$, or else $v_i(t)$ will be updated, which is again a contradiction.

However, at each stage $t \in (t^*, t_k]$, at least one update occurs. Therefore
\[
	t_k - t^* \le (r-1 - (r-k))(s-1) = (k-1)(s-1)
\]
and therefore $t_k \le t^* + (k-1)(s-1)$, proving~\eqref{eq:late-game}.

\textbf{Case 2:} $i \ge r-k+1$.

We make the substitution $h = r-1-i$ and $c^* = s-1-b^*$; note that $h \ge 0$ and $c^* \ge 0$. Since $i \ge r-k+1$, we have $h \le k-2$.

Setting $j=i$ in~\eqref{eq:product-bound}, we get $(r-i)(s-1-b^*) \le k-1$, or 
\[
	(h+1)c^* \le k-1.
\]
If $c^* = 0$, we have $h + c^* \le k-2 < k-1$; if $c^* \ge 1$, then $h + c^* \le hc^* + c^* \le k-1$. Therefore we always have $h + c^* \le k-1$.

By Lemma~\ref{lemma:increments},
\begin{align*}
t_k &\le t^* + i(s-1-b^*) + (r-1-i)b^* 
	= t^* + (r-1-h)c^* + h(s-1-c^*)\\
	&\le t^* + (s-1-h)c^*+ h(s-1-c^*)
	= t^* + (c^* + h)(s-1) - 2hc^*\\
	&\le t^* + (k-1)(s-1) - 0
\end{align*}
and we have shown \eqref{eq:late-game} again.
\end{proof}

\section{Conclusion}

The results of Theorems~\ref{thm:onlineupperbound} and~\ref{thm:onlinelowerbound} reduce the gap between the bounds on $r_o(P_n, P_n)$, but there is still a considerable gap here; the lower bound is $\Omega(n \log n)$ while the upper bound is $O(n^2 \log n)$. The most natural direction for further study is to ask: which of these bounds is closer to the truth?

Additionally, when the Erd\H os--Szekeres Theorem is interpreted as a statement about monotone subsequences, there is a corresponding online version of the question: how many comparisons need to be done on a sequence to find a monotone increasing subsequence of length $r+1$ or a monotone decreasing subsequence of length $s+1$? The worst-case analysis of this problem is a variant of the online size Ramsey number $r_o(P_r, P_s)$ that puts an additional limitation on Painter: the results of comparisons must obey transitivity.

In principle, there could be a gap between the number of comparisons in this problem, $r_o(P_r,P_s)$, and our variant $r_o^*(P_r, P_s)$ in which the number of vertices is fixed. However, all of our bounds apply to these problems equally. It would be interesting to determine if these problems do in fact have the same answer for all $r$ and $s$.

\section*{Acknowledgements}
We thank the anonymous referees for their many useful comments and suggestions. This research was performed while the third and fourth authors were at the University of Illinois Urbana-Champaign.

\bibliographystyle{abbrv}
\bibliography{ordered}

\begin{thebibliography}{10}

\bibitem{BD}
D.~Bal and L.~DeBiasio.
\newblock New lower bounds on the size-{R}amsey number of a path.
\newblock {\em arXiv:1909.06354}, 2019.

\bibitem{orderedbalogh}
J.~Balogh, F.~C. Clemen, E.~Heath, and M.~Lavrov.
\newblock Ordered size {R}amsey number of paths.
\newblock {\em Discrete Appl. Math.}, 276:13--18, 2020.

\bibitem{B1}
J.~Beck.
\newblock On size {R}amsey number of paths, trees, and circuits. {I}.
\newblock {\em Journal of Graph Theory}, 7:115--129, 1983.

\bibitem{B3}
J.~Beck.
\newblock Achievement games and the probabilistic method.
\newblock In {\em Combinatorics, {P}aul {E}rd\H{o}s is eighty, {V}ol. 1},
  Bolyai Soc. Math. Stud., pages 51--78. J\'{a}nos Bolyai Math. Soc., Budapest,
  1993.

\bibitem{BKS}
I.~Ben-Eliezer, M.~Krivelevich, and B.~Sudakov.
\newblock The size {R}amsey number of a directed path.
\newblock {\em J. Comb. Theory, Ser. B}, 102:743--755, 2010.

\bibitem{BLS}
M.~Buci\'c, S.~Letzter, and B.~Sudakov.
\newblock Monochromatic paths in random tournaments.
\newblock {\em Random Structures \& Algorithms}, 54(1):69--81, 2018.

\bibitem{BST}
M.~Buci{\'c}, B.~Sudakov, and T.~Tran.
\newblock Erd{\H o}s-{S}zekeres theorem for multidimensional arrays.
\newblock {\em arXiv:1910.13318}, 2019.

\bibitem{DP2}
A.~Dudek and P.~Pra\l{}at.
\newblock On some multicolour {R}amsey properties of random graphs.
\newblock {\em SIAM J.~Discrete Math.}, 31(3):2017--2092, 2017.

\bibitem{E}
P.~Erd\H{o}s.
\newblock On the combinatorial problems which {I} would most like to see
  solved.
\newblock {\em Combinatorica}, 1:25--42, 1981.

\bibitem{ESz}
P.~Erd\H{o}s and {\relax Gy}.~Szekeres.
\newblock A combinatorial problem in geometry.
\newblock {\em Compos. Math.}, 2:463--470, 1935.

\bibitem{FG}
P.~C. Fishburn and R.~L. Graham.
\newblock Lexicographic {R}amsey theory.
\newblock {\em Journal of Combinatorial Theory, Series A}, 62(2):280--298,
  1993.

\bibitem{FPSS}
J.~Fox, J.~Pach, B.~Sudakov, and A.~Suk.
\newblock Erd{\H{o}}s--{S}zekeres-type theorems for monotone paths and convex
  bodies.
\newblock {\em Proceedings of the London Mathematical Society},
  105(5):953--982, 2012.

\bibitem{GKP}
J.~A. Grytczuk, H.~A. Kierstead, and P.~Pra\l{}at.
\newblock On-line {R}amsey numbers for paths and stars.
\newblock {\em Discrete Math. Theor. Comput. Sci.}, 10(3):63--74, 2008.

\bibitem{KR}
A.~Kurek and A.~Ruci\'{n}ski.
\newblock Two variants of the size {R}amsey number.
\newblock {\em Discuss. Math. Graph Theory}, 25(1-2):141--149, 2005.

\bibitem{LS}
S.~Letzter and B.~Sudakov.
\newblock The oriented size {R}amsey number of directed paths.
\newblock {\em European J. Combin.}, 88:103103, 5, 2020.

\bibitem{PPW}
X.~P{\'e}rez-Gim{\'e}nez, P.~Pra\l{}at, and D.~B. West.
\newblock On-line size {R}amsey number for monotone {$k$}-uniform ordered paths
  with uniform looseness.
\newblock {\em European J. Combin.}, 92:Paper No. 103242, 14, 2021.

\bibitem{S}
J.~M. Steele.
\newblock Variations on the monotone subsequence theme of {E}rd{\H o}s and
  {S}zekeres.
\newblock In {\em Discrete Probability and Algorithms}, pages 111--131.
  Springer, 1995.

\end{thebibliography}

\end{document}